\documentclass[a4j,12pt,draft]{article}
\usepackage{indent first}
\usepackage[a4paper,vmargin={1in,1in},hmargin={1in,1in}]{geometry}
\usepackage{amsmath}
\usepackage{amssymb}
\usepackage{amsthm}
\usepackage{mathrsfs}
\usepackage{titlesec}
\titleformat{\section}{\large\bfseries}{\thesection}{1em}{}
\numberwithin{equation}{section}

\theoremstyle{definition}
\newtheorem{definition}{Definition}[section]
\theoremstyle{plain}
\newtheorem{thm}{Theorem}
\theoremstyle{plain}
\newtheorem{prop}{Proposition}[section]
\theoremstyle{plain}
\newtheorem{lem}[prop]{Lemma}
\theoremstyle{plain}
\newtheorem{cor}[thm]{Corollary}
\theoremstyle{definition}
\newtheorem{rmk}{Remark}

\allowdisplaybreaks
\newcommand\numberthis{\addtocounter{equation}{1}\tag{\theequation}}

\begin{document}

\begin{center}
\large{\textbf{On the zeros of the $k$-th derivative of the Riemann zeta function under the Riemann hypothesis}}
\end{center}

\vspace{2mm}
\begin{center}
\large{Ade Irma Suriajaya}
\let\thefootnote\relax
\let\thefootnote\relax
\let\thefootnote\relax\footnote{This work was partly supported by Nitori International Scholarship Foundation and the Iwatani Naoji Foundation.}
\end{center}

\begin{center}
Graduate School of Mathematics, Nagoya University,\\
Furo-cho, Chikusa-ku, Nagoya, 464-8602, Japan

\vspace{1mm}
\small{\textit{m12026a@math.nagoya-u.ac.jp}}
\end{center}

\vspace{1mm}
\begin{abstract}
The number of zeros and the distribution of the real part of non-real zeros of the derivatives of the Riemann zeta function have been investigated by Berndt, Levinson, Montgomery, and Akatsuka. Berndt, Levinson, and Montgomery studied the general case, meanwhile Akatsuka gave sharper estimates for the first derivative of the Riemann zeta function under the truth of the Riemann hypothesis. In this paper, we generalize the results of Akatsuka to the $k$-th derivative (for positive integer $k$) of the Riemann zeta function.
\end{abstract}

\vspace{1mm}
\section{Introduction}
\label{sec:1}

The theory of the Riemann zeta function $\zeta(s)$ has been studied for over 150 years. Among the topics of research, the study of its zeros has been one of the main subject of interest, in particular, the study of the zeros of its derivatives has also been part of the research area.
In fact, in 1970, Berndt \cite[Theorem]{ber} proved that
\begin{equation} \label{eq:nkt}
N_k(T) = \frac{T}{2\pi}\log{\frac{T}{4\pi}}-\frac{T}{2\pi}+O(\log{T})
\end{equation}
where $N_k(T)$ denotes the number of zeros of the $k$-th derivative of the Riemann zeta function $\zeta^{(k)}(s)$, with $0<\text{Im}\,(s)\leq T$, counted with multiplicity, for any positive integer $k$.
And in 1974, Levinson and Montgomery \cite[Theorem 10]{lev} showed that for any positive integer $k$,
\begin{equation} \label{eq:kdist}
\begin{aligned}
\sum_{\substack{\rho^{(k)}=\beta^{(k)}+i\gamma^{(k)},\\ \zeta^{(k)}(\rho^{(k)})=0,\, 0<\gamma^{(k)}\leq T}} \left(\beta^{(k)}-\frac{1}{2}\right) = \frac{kT}{2\pi}\log{\log{\frac{T}{2\pi}}} &+ \frac{1}{2\pi}\left(\frac{1}{2}\log{2} - k\log{\log{2}}\right)T \\
&- k\operatorname{Li}\left(\frac{T}{2\pi}\right) + O(\log{T})
\end{aligned}
\end{equation}
where the sum is counted with multiplicity and
\begin{align*}
\operatorname{Li}(x) := \int_2^x \frac{dt}{\log{t}}.
\end{align*}
In addition to the above result \eqref{eq:kdist}, Levinson and Montgomery \cite{lev} also studied the location of the zeros of $\zeta^{(k)}(s)$.
There are many other papers on the zeros of $\zeta^{(k)}(s)$; for example, Conrey and Ghosh \cite[Theorem 1]{con} in 1989, studied the zeros of $\zeta^{(k)}(s)$ near the critical line.

In 2012, Akatsuka \cite[Theorems 1 and 3]{aka} improved each of the error term of the results obtained by Berndt and by Levinson and Montgomery mentioned above (see \eqref{eq:nkt} and \eqref{eq:kdist}) for the case $k=1$ under the assumption of the truth of the Riemann hypothesis.
More precisely, he showed that
\begin{align*}
\sum_{\substack{\rho'=\beta'+i\gamma',\\ \zeta'(\rho')=0,\, 0<\gamma'\leq T}} \left(\beta'-\frac{1}{2}\right)
= \frac{T}{2\pi}\log{\log{\frac{T}{2\pi}}} &+ \frac{1}{2\pi}\left(\frac{1}{2}\log{2} - \log{\log{2}}\right)T \\
&- \operatorname{Li}\left(\frac{T}{2\pi}\right) + O((\log{\log{T})^2})
\end{align*}
and
\begin{equation*}
N_1(T) = \frac{T}{2\pi}\log{\frac{T}{4\pi}}-\frac{T}{2\pi}+O\left(\frac{\log{T}}{(\log{\log{T}})^{1/2}}\right)
\end{equation*}
if the Riemann hypothesis is true.
In this paper, we generalize these two results of Akatsuka for any positive integer $k$.
Before we introduce our results, we define some notation.

\vspace{3mm}
We denote by $\mathbb{Z}$, $\mathbb{R}$, and $\mathbb{C}$
the set of all rational integers, the set of all real numbers, and the set of all complex numbers, respectively.
Throughout this paper, the letter $k$ is used as a fixed positive integer, unless otherwise specified.
Next, let $\rho = \beta + i\gamma$ and $\rho^{(k)} = \beta^{(k)} + i\gamma^{(k)}$ be the nontrivial zeros of the Riemann zeta function and the non-real zeros of the $k$-th derivative of the Riemann zeta function, respectively. Then we define $N(T)$ and $N_k(T)$ as follows:

\begin{definition}
For $T>0$, we define
\begin{equation*}
N(T) := \sharp'\{\rho=\beta+i\gamma \,|\, 0<\gamma\leq T\}
\end{equation*}
and
\begin{equation*}
N_k(T) := \sharp'\{\rho^{(k)}=\beta^{(k)}+i\gamma^{(k)} \,|\, 0<\gamma^{(k)}\leq T\}
\end{equation*}
where $\sharp'$ means the number of elements counted with multiplicity.
\end{definition}

\vspace{2mm}
The following results generalize of Theorem 1, Corollary 2, and Theorem 3 of \cite{aka}, respectively. Note that each sum counts the non-real zeros of $\zeta^{(k)}(s)$ with multiplicity and that the implicit constant in $O_k(\cdot)$ depends only on $k$.

\begin{thm} \label{cha1}
Assume that the Riemann hypothesis is true. Then for any $T>2\pi$, we have
\begin{equation*}
\begin{aligned}
\sum_{\substack{\rho^{(k)}=\beta^{(k)}+i\gamma^{(k)},\\ 0<\gamma^{(k)}\leq T}} \left(\beta^{(k)} -\frac{1}{2}\right) = \frac{kT}{2\pi}\log{\log{\frac{T}{2\pi}}} &+ \frac{1}{2\pi}\left(\frac{1}{2}\log{2} - k\log{\log{2}}\right)T \\
&- k\operatorname{Li}\left(\frac{T}{2\pi}\right) + O_k((\log{\log{T}})^2).
\end{aligned}
\end{equation*}
\end{thm}

\begin{cor} [Cf. {\cite[Theorem 3]{lev}}] \label{cha12}
Assume that the Riemann hypothesis is true. Then for $0<U<T$ (where $T$ is restricted to satisfy $T>2\pi$), we have
\begin{equation*}
\begin{aligned}
\sum_{\substack{\rho^{(k)}=\beta^{(k)}+i\gamma^{(k)},\\ T<\gamma^{(k)}\leq T+U}} \left(\beta^{(k)}-\frac{1}{2}\right) = \frac{kU}{2\pi}\log{\log{\frac{T}{2\pi}}} &+ \frac{1}{2\pi}\left(\frac{1}{2}\log{2} - k\log{\log{2}}\right)U \\
&+ O\left(\frac{U^2}{T\log{T}}\right) + O_k((\log{\log{T}})^2).
\end{aligned}
\end{equation*}
Here the implicit constant in the error term $O\left(\frac{U^2}{T\log{T}}\right)$ does not depend on any parameter.
\end{cor}

\begin{thm} \label{cha2}
Assume that the Riemann hypothesis is true. Then for $T\geq3$, we have
\begin{equation*}
N_k(T) = \frac{T}{2\pi}\log{\frac{T}{4\pi}} - \frac{T}{2\pi} + O_k\left(\frac{\log{T}}{(\log{\log{T}})^{1/2}}\right).
\end{equation*}
\end{thm}

\vspace{2mm}
We write Re$(s)$ and Im$(s)$ (for any $s\in\mathbb{C}$) as $\sigma$ and $t$, respectively. We abbreviate the Riemann hypothesis as RH, and finally,
we define two functions $F(s)$ and $G_k(s)$ as follows:
\begin{definition}
\begin{align*}
F(s) := 2^s\pi^{s-1}\sin{\left(\frac{\pi s}{2}\right)}\Gamma(1-s),\quad G_k(s) := (-1)^k\frac{2^s}{(\log{2})^k}\zeta^{(k)}(s).
\end{align*}
\end{definition}
\noindent
By the above definition of $F(s)$, we can check easily that the functional equation for $\zeta(s)$ states
\begin{equation} \label{eq:fe}
\zeta(s) = F(s)\zeta(1-s).
\end{equation}

\begin{rmk}
The function $F(s)$ appeared in \cite{aka} and \cite[Section 3]{lev} and the function $G_k(s)$ is the $\zeta^{(k)}$-version of the function $G(s)$ in \cite{aka}, which is denoted by $Z_k(s)$ in \cite[Section 3]{lev}. Most of the symbols used in this paper follow those used in \cite{aka}.
\end{rmk}

Before we move to the next section, we intend to give a brief outline of the proofs.
Nevertheless, since the steps of our proofs basically follow those given in \cite{aka} with a few crucial modifications, instead of the outline of the proofs, we present only the main needed modifications related to the proofs.

\vspace{2mm}
First of all, condition 2 of Lemma 2.1 of \cite{aka} is related to the functional equation for $\zeta'(s)$. In our case, we need to consider $\zeta^{(k)}(s)$ for any positive integer $k$. Thus, we obtain a function which consists of terms that are not logarithmic derivatives of some functions so we cannot easily follow the case of $\zeta'(s)$. In the present paper, we take care of these terms in a way that does not involve any calculation on logarithmic derivatives.

Secondly, similar to condition 2, in condition 3 of Lemma 2.1 of \cite{aka}, the factor to be estimated was $\frac{F'}{F}(s)$ which is just the logarithmic derivative of $F(s)$, whereas in the present paper, we need to take care of $\frac{F^{(k)}}{F}(s)$ which is not a logarithmic derivative of any function. Thus, as in condition 2, we estimate this term for any $k$ in a way which does not require any calculation on logarithmic derivatives, and hence we need to take a suitable logarithmic branch of the function $\log{\frac{F^{(k)}}{F}(s)}$.

The next is condition 4 of Lemma 2.1 of \cite{aka}. For $\zeta'(s)$, the term we need to estimate was $\frac{\zeta'}{\zeta}(s)$ which is just the logarithmic derivative of $\zeta(s)$. In \cite{aka}, the inequality $\text{Re}\,\left(\frac{\zeta'}{\zeta}(s)\right)<0$ was obtained, however for $\zeta^{(k)}(s)$, the sign of $\text{Re}\,\left(\frac{\zeta^{(k)}}{\zeta}(s)\right)$ does not seem to stay unchanged in any region defined by $x\leq\sigma<\frac{1}{2},\, t\geq y$ for some $x\leq-1$ and large $y>0$. Nevertheless, since it is sufficient to show that $\frac{\zeta^{(k)}}{\zeta}(s)$ is holomorphic and non-zero, and has bounded argument in some region of the above kind, we shall modify the condition in such a way.

Furthermore, with the modifications of these conditions of the first lemma, the choice of logarithmic branch of the function $\log{\left(\frac{1}{\frac{F^{(k)}}{F}(s)}\frac{\zeta^{(k)}}{\zeta}(s)\right)}$ in the proof of Proposition \ref{prop2} (which generalizes Proposition 2.2 in \cite{aka}) must be taken more carefully so that these conditions can be used in our calculations. In order to evaluate the function $\log{\left(\frac{1}{\frac{F^{(k)}}{F}(s)}\frac{\zeta^{(k)}}{\zeta}(s)\right)}$, we first define the functions $\log{\left(\frac{1}{\frac{F^{(k)}}{F}(s)}\frac{\zeta^{(k)}}{\zeta}(s)\right)}$, $\log{\frac{F^{(k)}}{F}(s)}$, and $\log{\frac{\zeta^{(k)}}{\zeta}(s)}$ independently. Then using the continuities of $\arg{\left(\frac{1}{\frac{F^{(k)}}{F}(s)}\frac{\zeta^{(k)}}{\zeta}(s)\right)}$, $\arg{\frac{F^{(k)}}{F}(s)}$, and $\arg{\frac{\zeta^{(k)}}{\zeta}(s)}$, we observe the difference
$$
\arg{\left(\frac{1}{\frac{F^{(k)}}{F}(s_0)}\frac{\zeta^{(k)}}{\zeta}(s)\right)} - \left(-\arg{\frac{F^{(k)}}{F}(s)} + \arg{\frac{\zeta^{(k)}}{\zeta}(s)}\right)
$$
in the region under evaluation (see the evaluation of $I_{15}$ in Proposition \ref{prop2}).

Finally, the region $\frac{1}{2}<\sigma\leq a$ considered in Lemma 2.3 of \cite{aka} does not work well for $\frac{\zeta^{(k)}}{\zeta}(s)$. The reason is that the current best estimate of $\frac{\zeta^{(k)}}{\zeta}(s)$ depends on the usage of Cauchy's integral formula, hence we need to keep a certain distance between $\frac{1}{2}$ and the infimum of $\sigma$ in the region. Therefore, we put here a small distance $\epsilon_0>0$ (see the statement of our Lemma \ref{lem3}).


\section{Proof of Theorem \ref{cha1} and Corollary \ref{cha12}}
\label{sec:2}

In this section we give the proofs of Theorem 1 and Corollary 2. For that purpose, we need a few lemmas and a proposition which are analogues of those in \cite{aka}.

\vspace{3mm}
The following lemma is a generalization of Lemma 2.1 of \cite{aka} for the case of $\zeta^{(k)}(s)$.

\begin{lem} \label{lem1}
Assume RH. Then there exist $a_k\geq10$, $\sigma_k\leq-1$, and $t_k\geq\max{\{a_k^2,-\sigma_k\}}$ such that the following conditions are satisfied:
\begin{enumerate}
\item[1.] $\displaystyle \left|G_k(s)-1\right| \leq \frac{1}{2}\left(\frac{2}{3}\right)^{\sigma/2} $,
\quad for any $\sigma\geq a_k$;
\item[2.] $\displaystyle \left|\sum_{j=1}^k {{k}\choose{j}} (-1)^{j-1} \frac{1}{\frac{F^{(k)}}{F^{(k-j)}}(s)}\frac{\zeta^{(j)}}{\zeta}(1-s)\right| \leq 2^\sigma $, \quad for $\sigma\leq\sigma_k$ and $t\geq2$;
\item[3.] $\displaystyle \left|\frac{F^{(k)}}{F}(s)\right|\geq1 $
holds in the region $\sigma_k\leq\sigma\leq\frac{1}{2},\, t\geq t_k-1$.
Furthermore, we can take the logarithmic branch of $\displaystyle \log{\frac{F^{(k)}}{F}(s)} $
in that region such that it is holomorphic there and
$$ \frac{\alpha_k\pi}{6} < \arg{\frac{F^{(k)}}{F}(s)} < \frac{\beta_k\pi}{6} $$
holds, where
\begin{align*}
(\alpha_k,\beta_k)=
\begin{cases}
(5,7) \text{ if } k\text{ is odd},\\
(-1,1) \text{ if } k\text{ is even};
\end{cases}
\end{align*}
\item[4.] $\displaystyle \frac{\zeta^{(k)}}{\zeta}(s)\neq0 $ holds in the region $\sigma_k\leq\sigma<\frac{1}{2},\, t\geq t_k-1$.
Furthermore, we can take the logarithmic branch of $\displaystyle \log{\frac{\zeta^{(k)}}{\zeta}(s)} $
in that region such that it is holomorphic there and
$$ \frac{k\pi}{2} < \arg{\frac{\zeta^{(k)}}{\zeta}(s)} < \frac{3k\pi}{2} $$
holds;
\item[5.] $\displaystyle \zeta(\sigma+it_k)\neq0,\, \zeta^{(k)}(\sigma+it_k)\neq0 $, \quad for all $\sigma\in\mathbb{R}$.
\end{enumerate}
\end{lem}

\begin{proof}

\begin{enumerate}
\item[1.] See \cite[(3.2) (p. 54)]{lev}.

\item[2.] We start by estimating $\frac{F^{(k)}}{F^{(k-j)}}(s)$ ($j=1,2,\cdots,k$) in the region $\sigma<1,\, t\geq2$.
We set $f(s):=\left(\frac{1}{2}-s\right)\left(\log{(1-s)} - \log{(2\pi)} + \frac{\pi i}{2} \right) + s + O(1)$, where $f(s)$ is an analytic function and
\begin{equation*}
f'(s) = -\log{(1-s)} + O(1),\quad\quad f^{(j)}(s) = O(1) \quad\quad (j\geq2).
\end{equation*}
As in \cite[pp. 54--55]{lev}, we can write
\begin{equation*}
F(s) = \exp(f(s)).
\end{equation*}
Using methods similar to \cite[Lemma 6 (p. 133)]{gon} and \cite[pp. 54--55]{lev},
we can show that
\begin{equation} \label{eq:Fj}
F^{(j)}(s) = F(s)(f'(s))^j\left(1+O\left(\frac{1}{|\log{s}|^2}\right)\right)
\end{equation}
holds for any positive integer $j$. In consequence, for $j=1,2,\cdots,k$, we have
\begin{align*}
\left|\frac{F^{(k)}}{F^{(k-j)}}(s)\right|
&= \left|(f'(s))^j\left(1+O\left(\frac{1}{|\log{s}|^2}\right)\right)\right| \\
&\geq (\log{|1-s|})^j - \left|O\left((\log{|1-s|})^{j-1}\right)\right|.
\end{align*}
Certainly, this also holds in the region $\sigma\leq-1,\, t\geq2$, so for any positive integer $k$, we can take $\sigma_{k_1}\leq-1$ sufficiently small (i.e. sufficiently large in the negative direction) so that for any $s$ with $\sigma\leq\sigma_{k_1}$ and $t\geq2$, we have
\begin{equation} \label{eq:Fkkj}
\left|\frac{F^{(k)}}{F^{(k-j)}}(s)\right| \geq \frac{1}{2k} (\log{|1-s|})^j \geq \frac{1}{2k}(\log{(1-\sigma)})^j.
\end{equation}

\hspace{5mm}
Next we estimate $\frac{\zeta^{(j)}}{\zeta}(1-s)$ ($j=1,2,\cdots,k$).
In the region $\sigma\leq-1,\, t\geq2$, we have
\begin{align*}
\left|\zeta^{(j)}(1-s)\right| &\leq \left|\frac{(\log{2})^j}{2^{1-s}}\right| + \left|\sum_{n=3}^\infty \frac{(\log{n})^j}{n^{1-s}}\right|
\leq \frac{1}{2}(\log{2})^j2^\sigma + \int_2^\infty \frac{(\log{x})^j}{x^{1-\sigma}}dx\\
&= 2^\sigma\left(\frac{1}{2}(\log{2})^j + \sum_{l=0}^j\frac{(\log{2})^{j-l}\frac{j!}{(j-l)!}}{(-\sigma)^{l+1}}\right)
\end{align*}
and
\begin{equation*}
\left|\zeta(1-s)\right| \geq 1 - \left|\sum_{n=2}^\infty \frac{1}{n^{1-s}}\right|
\geq 1 - \sum_{n=2}^\infty \frac{1}{n^2} = 2 - \frac{\pi^2}{6}.
\end{equation*}
Thus,
\begin{equation} \label{eq:zetaj}
\left|\frac{\zeta^{(j)}}{\zeta}(1-s)\right| \leq \frac{2^\sigma}{2 - \frac{\pi^2}{6}}\left(\frac{1}{2}(\log{2})^j + \sum_{l=0}^j\frac{(\log{2})^{j-l}\frac{j!}{(j-l)!}}{(-\sigma)^{l+1}}\right).
\end{equation}

\hspace{5mm}
Now combining \eqref{eq:Fkkj} and \eqref{eq:zetaj}, for $\sigma\leq\sigma_{k_1}$ and $t\geq2$, we have
\begin{align*}
&\left|\sum_{j=1}^k {{k}\choose{j}} (-1)^{j-1} \frac{1}{\frac{F^{(k)}}{F^{(k-j)}}(s)}\frac{\zeta^{(j)}}{\zeta}(1-s)\right| \\
&\quad\quad\quad\quad\quad
\leq \sum_{j=1}^k {{k}\choose{j}} \frac{1}{\left|\frac{F^{(k)}}{F^{(k-j)}}(s)\right|} \left|\frac{\zeta^{(j)}}{\zeta}(1-s)\right| \\
&\quad\quad\quad\quad\quad
\leq 2^\sigma\frac{2k}{2 - \frac{\pi^2}{6}} \sum_{j=1}^k {{k}\choose{j}} \frac{1}{(\log{(1-\sigma)})^j} \left(\frac{1}{2}(\log{2})^j + \sum_{l=0}^j\frac{(\log{2})^{j-l}\frac{j!}{(j-l)!}}{(-\sigma)^{l+1}}\right).
\end{align*}

\hspace{5mm}
Since for any positive integer $k$,
\begin{align*}
\lim_{\sigma\rightarrow-\infty} \frac{2k}{2 - \frac{\pi^2}{6}} \sum_{j=1}^k {{k}\choose{j}} \frac{1}{(\log{(1-\sigma)})^j} \left(\frac{1}{2}(\log{2})^j + \sum_{l=0}^j\frac{(\log{2})^{j-l}\frac{j!}{(j-l)!}}{(-\sigma)^{l+1}}\right) = 0,
\end{align*}
we can take $\sigma_k\leq\sigma_{k_1}$ ($\leq-1$) so that
\begin{align*}
\frac{2k}{2 - \frac{\pi^2}{6}} \sum_{j=1}^k {{k}\choose{j}} \frac{1}{(\log{(1-\sigma)})^j} \left(\frac{1}{2}(\log{2})^j + \sum_{l=0}^j\frac{(\log{2})^{j-l}\frac{j!}{(j-l)!}}{(-\sigma)^{l+1}}\right) \leq 1
\end{align*}
holds for any $\sigma\leq\sigma_k$.
This implies that
\begin{align*}
 \left|\sum_{j=1}^k {{k}\choose{j}} (-1)^{j-1} \frac{1}{\frac{F^{(k)}}{F^{(k-j)}}(s)}\frac{\zeta^{(j)}}{\zeta}(1-s)\right| \leq 2^\sigma
\end{align*}
holds for $\sigma\leq\sigma_k$, $t\geq2$.

\vspace{2mm} \hspace{5mm}
Now with the above $\sigma_k$, we are going to find $t_k\geq\max{\{a_k^2, -\sigma_k\}}$ for which conditions 3 to 5 hold.

\item[3.] We start by examining condition 3. We first consider the region $\sigma_k\leq\sigma\leq\frac{1}{2},\, t\geq99$.
It follows from \eqref{eq:Fj} that in this region,
\begin{equation} \label{eq:Fk}
F^{(k)}(s) = F(s)(-\log{(1-s)} + O(1))^k\left(1+O\left(\frac{1}{|\log{s}|^2}\right)\right)
\end{equation}
holds.
This gives us,
\begin{align*}
\left|\frac{F^{(k)}}{F}(s)\right| \geq |(\log{(1-s)})^k| - \left|O_{\sigma_k}((\log{t})^{k-1})\right| \geq (\log{t})^k - \left|O_{\sigma_k}((\log{t})^{k-1})\right|
\end{align*}
for $\sigma_k\leq\sigma\leq\frac{1}{2}$ and $t\geq99$.
Thus, for any integer $k\geq1$, we can take $t_{k_1}\geq100$ such that
\begin{equation} \label{eq:cond31}
\left|\frac{F^{(k)}}{F}(s)\right| \geq 1
\end{equation}
for $\sigma_k\leq\sigma\leq\frac{1}{2}$ and $t\geq t_{k_1}-1$.

\hspace{5mm}
We note from \eqref{eq:Fk} that $\frac{F^{(k)}}{F}(s) = (-1)^k(\log{t})^k + O((\log{t})^{k-1})$ when $\sigma_k\leq\sigma\leq\frac{1}{2}$ and $t\geq99$.
Consequently, for odd integer $k\geq1$, we can find sufficiently large $t'_{k_2}\geq100$ such that
\begin{equation*}
\frac{5\pi}{6} < \arg{\frac{F^{(k)}}{F}(s)} < \frac{7\pi}{6}
\end{equation*}
holds for $\sigma_k\leq\sigma\leq\frac{1}{2}$ and $t\geq t'_{k_2}-1$.
Similarly, when $k$ is even, we can also find sufficiently large $t''_{k_2}\geq100$ such that
\begin{equation*}
-\frac{\pi}{6} < \arg{\frac{F^{(k)}}{F}(s)} < \frac{\pi}{6}
\end{equation*}
holds for $\sigma_k\leq\sigma\leq\frac{1}{2}$ and $t\geq t''_{k_2}-1$.
Since all zeros and poles of $F(s)$ lie on $\mathbb{R}$, $\frac{F^{(k)}}{F}(s)$ has no poles for $t>0$. This along with \eqref{eq:cond31} implies that $\log{\frac{F^{(k)}}{F}(s)}$ is holomorphic in the region with this branch.
Thus setting
\begin{align*}
(\alpha_k,\beta_k) :=
\begin{cases}
(5,7), &\quad \text{if}\,\, k \,\,\text{is odd},\\
(-1,1), &\quad \text{if}\,\, k \,\,\text{is even};
\end{cases}
\end{align*}
and
\begin{align*}
t_{k_2} :=
\begin{cases}
t'_{k_2}, &\quad \text{if}\,\, k \,\,\text{is odd},\\
t''_{k_2}, &\quad \text{if}\,\, k \,\,\text{is even};
\end{cases}
\end{align*}
we find that $\log{\frac{F^{(k)}}{F}(s)}$ is holomorphic and that
\begin{equation*}
\frac{\alpha_k\pi}{6} < \arg{\frac{F^{(k)}}{F}(s)} < \frac{\beta_k\pi}{6}
\end{equation*}
holds in the region $\sigma_k\leq\sigma\leq\frac{1}{2},\, t\geq t_{k_2}-1$.

\vspace{2mm} \hspace{5mm}
By the above calculations, we see that $\max{\{t_{k_1}, t_{k_2}, a_k^2, -\sigma_k\}}$ is a candidate for $t_k$. Thus we have proven that $t_k\geq\max{\{a_k^2, -\sigma_k\}}$ for which condition 3 holds exists. Since we want $t_k$ to also satisfy conditions 4 and 5, we need to examine those conditions to completely prove the existence of $t_k$.

\item[4.] Referring to \cite[Corollary of Theorem 7 (p. 51)]{lev}, we know that RH implies that for any positive integer $j$, $\zeta^{(j)}(s)$ has at most a finite number of non-real zeros in $\sigma<\frac{1}{2}$. Hence we can number all the non-real zeros of $\zeta^{(j)}(s)$ in $\sigma<\frac{1}{2}$ as $\rho_1^{(j)}, \rho_2^{(j)}, \rho_3^{(j)}, \cdots, \rho_{m_j}^{(j)}$ ($\rho_l^{(j)} = \beta_l^{(j)} + i\gamma_l^{(j)}$) for some integer $m_j\geq2$ (note that if $\zeta^{(j)}\left(\rho^{(j)}\right) = 0$, then $\zeta^{(j)}\left(\overline{\rho^{(j)}}\right) = 0$, so $m_j\geq2$) in the order such that $\gamma_l^{(j)}\leq\gamma_{l+1}^{(j)}$ for all $1\leq l\leq m_j-1$.
Therefore, $\zeta^{(j)}(s)\neq0$ when $\sigma<\frac{1}{2}$ and $t\geq\gamma_{m_j}^{(j)}+1$.
We set $t_{k_3} := \displaystyle{\max_{1\leq j\leq k}{(\gamma_{m_j}^{(j)}+2)}}$, then for all $j=1,2,\cdots,k$, we have
\begin{equation} \label{eq:4zenhan}
\zeta^{(j)}(s)\neq0
\end{equation}
in the region $\sigma<\frac{1}{2},\, t\geq t_{k_3}-1$.

\hspace{5mm}
Next we show that we can take the logarithmic branch of $\log{\frac{\zeta^{(k)}}{\zeta}(s)}$ in the region $\sigma_k\leq\sigma<\frac{1}{2},\, t\geq t_{k_4}-1$ for some $t_{k_4}\geq100$, so that it is holomorphic there and
\begin{align*}
\frac{k\pi}{2} < \arg{\frac{\zeta^{(k)}}{\zeta}(s)} < \frac{3k\pi}{2}
\end{align*}
holds there by first claiming that we can find some $t_{k_4}\geq t_{k_3}$ for which
\begin{equation} \label{eq:claim1}
\text{Re}\left(\frac{\zeta^{(j)}}{\zeta^{(j-1)}}(s)\right) < 0 \quad\quad (\sigma_k\leq\sigma<\frac{1}{2},\, t\geq t_{k_4}-1)
\end{equation}
holds for all $j=1,2,\cdots,k$.
We first note that for any $j=1,2,\cdots,k$, $\frac{\zeta^{(j)}}{\zeta^{(j-1)}}(s)$ is holomorphic and has no zeros in the region defined by $\sigma<\frac{1}{2}$ and $t\geq t_{k_3}-1$.

\hspace{5mm}
To show this, we refer to \cite[pp. 64--65]{lev} and we can show that for any $j=1,2,\cdots,k$,
\begin{align*}
\text{Re}\left(\frac{\zeta^{(j)}}{\zeta^{(j-1)}}(s)\right) \leq -\frac{2}{9}\log{|s|} + O_{\sigma_k}(1)
\end{align*}
holds when $\sigma_k\leq\sigma<\frac{1}{2}$, and $ t\geq t_{k_3}-1$.
Thus, we can take $t_{k_4}\geq t_{k_3}$ such that \eqref{eq:claim1} holds for all $j=1,2,\cdots,k$ .

\hspace{5mm}
The above immediately implies that for each $j=1,2,\cdots,k$, there exists an integer $l_j$ such that
\begin{equation} \label{eq:forlj}
\frac{\pi}{2}+2l_j\pi<\arg{\frac{\zeta^{(j)}}{\zeta^{(j-1)}}(s)}<\frac{3\pi}{2}+2l_j\pi
\end{equation}
holds for $\sigma_k\leq\sigma<\frac{1}{2},\, t\geq t_{k_4}-1$.
We then choose the logarithmic branch of each of $\log{\frac{\zeta^{(j)}}{\zeta^{(j-1)}}(s)}$ such that each $l_j$ in \eqref{eq:forlj} is zero and take the logarithmic branch of $\log{\frac{\zeta^{(k)}}{\zeta}(s)}$ so that
\begin{equation*}
\arg{\frac{\zeta^{(k)}}{\zeta}(s)} = \sum_{j=1}^k \arg{\frac{\zeta^{(j)}}{\zeta^{(j-1)}}(s)}
\end{equation*}
holds in the region $\sigma_k\leq\sigma<\frac{1}{2},\, t\geq t_{k_4}-1$.
Note that from \eqref{eq:4zenhan} and the analyticity of $\zeta^{(k)}(s)$ in $\sigma<\frac{1}{2}$ (also note that we are assuming RH thus $\zeta(s)\neq0$ when $\sigma<\frac{1}{2}$ and $t\geq t_{k_4}-1$), $\log{\frac{\zeta^{(k)}}{\zeta}(s)}$ is holomorphic in this region with this branch.
We then obtain a holomorphic function $\log{\frac{\zeta^{(k)}}{\zeta}(s)}$ with inequalities
\begin{equation*}
\frac{k\pi}{2} < \arg{\frac{\zeta^{(k)}}{\zeta}(s)} < \frac{3k\pi}{2}
\end{equation*}
in the region $\sigma_k\leq\sigma<\frac{1}{2},\, t\geq t_{k_4}-1$.

\vspace{2mm} \hspace{5mm}
Combining the proof of condition 3 and the above calculations, we find that $\max\{t_{k_1},$ $t_{k_2}, t_{k_4}, a_k^2, -\sigma_k\}$ is a candidate for $t_k$. Therefore we have proven that $t_k\geq\max{\{a_k^2, -\sigma_k\}}$ for which conditions 3 and 4 hold exists.

\item[5.] Now we set $t_{k_5}:=\max{\{t_{k_1}, t_{k_2}, t_{k_4}, a_k^2, -\sigma_k\}}$.
\begin{itemize}
\item Since we are assuming RH, $\zeta(\sigma+it)\neq0$ for any $t>0$ if $\sigma\neq\frac{1}{2}$.
\item According to \cite[Table 1 (p. 678)]{spi2}, $\zeta'(\sigma+it)\neq0$ for any $t\in\mathbb{R}$ if $\sigma\geq3$ and $\zeta''(\sigma+it)\neq0$ for any $t\in\mathbb{R}$ if $\sigma\geq5$.
According to \cite[Theorem 1]{spi2}, for $k\geq3$, $\zeta^{(k)}(\sigma+it)\neq0$ for any $t\in\mathbb{R}$ if $\sigma\geq\frac{7}{4}k+2$.
Indeed, we can check that for $k=1$, $\frac{7}{4}+2>3$ and for $k=2$, $\frac{7}{2}+2>5$,
thus for any positive integer $k$,
\begin{align*}
\zeta^{(k)}(\sigma+it)\neq0 \quad\quad (\sigma\geq\frac{7}{4}k+2,\, t\in\mathbb{R}).
\end{align*}
\item Since $t_{k_5}\geq t_{k_3}$, from \eqref{eq:4zenhan}, we have
$\zeta^{(k)}(\sigma+it)\neq0$ for $\sigma<\frac{1}{2}$ and $t\geq t_{k_5}$.
\end{itemize}

Hence, for any positive integer $k$, we only need to find $t_k\in[t_{k_5}+1,t_{k_5}+2]$ for which
\begin{align*}
\zeta\left(\frac{1}{2}+it_k\right)\neq0 \quad\quad \text{and} \quad\quad \zeta^{(k)}(\sigma+it_k)\neq0\quad \text{for } \frac{1}{2} \leq \sigma \leq \frac{7}{4}k+2
\end{align*}
hold.
Note that this is possible by the identity theorem for complex analytic functions.
Thus, we have shown that $t_k$ defined above satisfies $t_k\geq\max{\{a_k^2,-\sigma_k\}}$ and also conditions 3 to 5.
\end{enumerate}
\end{proof}

\begin{rmk}
For $k=1$ and $k=2$, more precise results are known. Refer to \cite{aka} and \cite{ade}, respectively.
These results are obtained based on the works of Speiser \cite{spe}, Spira \cite{spi3}, and Yildirim \cite{yil1} (also \cite{yil2}) on the zeros of $\zeta'(s)$ and $\zeta''(s)$.
\end{rmk}


\begin{prop} \label{prop2}
Assume RH. Take $a_k$ and $t_k$ which satisfy all conditions of Lemma \ref{lem1}.
Then for $T\geq t_k$ which satisfies $\zeta^{(k)}(\sigma+iT)\neq0$ and $\zeta(\sigma+iT)\neq0$ for any $\sigma\in\mathbb{R}$, we have
\begin{align*}
\sum_{\substack{\rho^{(k)}=\beta^{(k)}+i\gamma^{(k)},\\ 0<\gamma^{(k)}\leq T}} \left(\beta^{(k)}-\frac{1}{2}\right) &= \frac{kT}{2\pi}\log{\log{\frac{T}{2\pi}}} + \frac{1}{2\pi}\left(\frac{1}{2}\log{2} - k\log{\log{2}}\right)T - k\operatorname{Li}\left(\frac{T}{2\pi}\right)\\
&\quad+ \frac{1}{2\pi}\int_{1/2}^{a_k} \left(-\arg{\zeta(\sigma+iT)}+\arg{G_k(\sigma+iT)}\right)d\sigma + O_k(1),
\end{align*}
where the logarithmic branches are taken so that $\log{\zeta(s)}$ and $\log{G_k(s)}$ tend to 0 as $\sigma\rightarrow\infty$ and are holomorphic in $\mathbb{C}\backslash\{\rho+\lambda \,|\, \zeta(\rho)=0 \,\,\text{or}\,\, \infty, \, \lambda\leq0\}$ and $\mathbb{C}\backslash\{\rho^{(k)}+\lambda \,|\, \zeta^{(k)}(\rho^{(k)})=0 \,\,\text{or}\,\, \infty, \, \lambda\leq0\}$, respectively.
\end{prop}

\begin{proof}
The steps of the proof generally follow the proof of Proposition 2.2 of \cite{aka}.
We first take $a_k$, $\sigma_k$, and $t_k$ as in Lemma \ref{lem1} and fix them.
Then, we take $T\geq t_k$ such that $\zeta^{(k)}(\sigma+iT)\neq0$ and $\zeta(\sigma+iT)\neq0$ ($^\forall\sigma\in\mathbb{R}$).
We also let $\delta\in(0,1/2]$ and put $b:=\frac{1}{2}-\delta$.
We consider the rectangle with vertices $b+it_k$,  $a_k+it_k$, $a_k+iT$, and $b+iT$, and then we apply Littlewood's lemma (cf. \cite[pp. 132--133]{tit1}) to $G_k(s)$ there.
By taking the imaginary part, we obtain
\begin{align*}
2\pi\sum_{\substack{\rho^{(k)}=\beta^{(k)}+i\gamma^{(k)},\\ t_k<\gamma^{(k)}\leq T}} (\beta^{(k)}-b) &=
\int_{t_k}^T\log{|G_k(b+it)|}dt - \int_{t_k}^T\log{|G_k(a_k+it)|}dt\\
&\quad- \int_b^{a_k}\arg{G_k(\sigma+it_k)}d\sigma + \int_b^{a_k}\arg{G_k(\sigma+iT)}d\sigma\\
&=: I_1 + I_2 + I_3 + \int_b^{a_k}\arg{G_k(\sigma+iT)}d\sigma \numberthis \label{eq:lit}
\end{align*}
where the sum is counted with multiplicity.
By the same reasoning as in \cite[p. 2246]{aka}, we have
\begin{equation*}
I_2 = O_{a_k}(1),\quad I_3 = O_{a_k,t_k}(1).
\end{equation*}

\vspace{2mm}
Now we only need to estimate $I_1$.
From the functional equation for $\zeta(s)$ (see \eqref{eq:fe}), we can deduce that
\begin{align*}
\zeta^{(k)}(s) &= F^{(k)}(s)\zeta(1-s)\left(1 - \sum_{j=1}^k {{k}\choose{j}} (-1)^{j-1} \frac{1}{\frac{F^{(k)}}{F^{(k-j)}}(s)}\frac{\zeta^{(j)}}{\zeta}(1-s)\right)\\
&= F(s)\frac{F^{(k)}}{F}(s)\zeta(1-s)\left(1 - \sum_{j=1}^k {{k}\choose{j}} (-1)^{j-1} \frac{1}{\frac{F^{(k)}}{F^{(k-j)}}(s)}\frac{\zeta^{(j)}}{\zeta}(1-s)\right).
\end{align*}
Hence,
\begin{align*}
I_1 &= \int_{t_k}^T\log{|G_k(b+it)|}dt = \int_{t_k}^T\log{\frac{2^b}{(\log{2})^k}|\zeta^{(k)}(b+it)|}dt\\
&= \int_{t_k}^T\log{\frac{2^b}{(\log{2})^k}}dt +  \int_{t_k}^T\log{|\zeta^{(k)}(b+it)|}dt\\
&= (b\log{2}-k\log{\log{2}})(T-t_k) + \int_{t_k}^T\log{|F(b+it)|}dt + \int_{t_k}^T\log{\left|\frac{F^{(k)}}{F}(b+it)\right|}dt\quad\\
&\quad+ \int_{t_k}^T\log{|\zeta(1-b-it)|}dt\\
&\quad+ \int_{t_k}^T\log{\left|1 - \sum_{j=1}^k {{k}\choose{j}} (-1)^{j-1} \frac{1}{\frac{F^{(k)}}{F^{(k-j)}}(b+it)}\frac{\zeta^{(j)}}{\zeta}(1-b-it)\right|}dt\\
&=: ((b\log{2}-k\log{\log{2}})T + O_{t_k}(1)) + I_{12} + I_{13} + I_{14} + I_{15}. \numberthis \label{eq:i1}
\end{align*}

\vspace{2mm}
As shown in \cite[pp. 2247--2249]{aka},
\begin{align*}
I_{12} = \left(\frac{1}{2}-b\right)\left(T\log{\frac{T}{2\pi}}-T\right) + O_{t_k}(1),
\end{align*}
\begin{equation*}
I_{14} = -\int_{1-b}^{a_k}\arg{\zeta(\sigma+iT)}d\sigma + O_{a_k,t_k}(1).
\end{equation*}

\vspace{2mm}
Below we estimate $I_{13}$ and $I_{15}$.

\vspace{2mm}
We begin with the estimation of $I_{13}$. We consider for $0<\sigma<\frac{1}{2}$ and $t\geq100$.
We first show that
\begin{equation} \label{eq:Fk2}
F^{(k)}(s) = F(s)(f'(s))^k\left(1 + O\left(\frac{e^{-t}}{|\log{s}|^2}\right) + O\left(\frac{1}{|s||\log{s}|^2}\right)\right)
\end{equation}
holds in the region $\sigma<1,\, t\geq100$. 
It is obvious that the above error estimate is more precise than that in \eqref{eq:Fj}. The proof is similar to the proof of condition 2 of Lemma \ref{lem1}.
We begin by taking the logarithmic branch of $\log{\left(\sin{\frac{\pi s}{2}}\right)}$ as
\begin{equation} \label{eq:logsin}
\log{\left(\sin{\frac{\pi s}{2}}\right)} = -\frac{\pi is}{2}-\log{2}+\frac{\pi i}{2}-\sum_{n=1}^\infty\frac{e^{\pi ins}}{n}
\end{equation}
in the region $0<\sigma<1,\, t\geq2$ and analytically continue it to the region $\sigma<1,\,$ $t\geq2$.
Next, we apply Stirling's formula to $\Gamma(1-s)$ in the region $-\frac{\pi}{2}<\arg{(1-s)}<\frac{\pi}{2}$. Substituting these into $F(s)$, we obtain
\begin{align*}
F(s) = \exp\left(\frac{\pi i}{4} - 1 + \left(\frac{1}{2}-s\right)\log{\frac{(1-s)i}{2\pi}} + s + O(e^{-t}) + O\left(\frac{1}{|s|}\right)\right)
\end{align*}
for $\sigma<1$ and $t\geq100$, where the term $O(e^{-t})$ comes from the term $\sum_{n=1}^\infty\frac{e^{\pi ins}}{n}$ in \eqref{eq:logsin} and the term $O\left(\frac{1}{|s|}\right)$ originates from the Stirling's formula.

We now write $f(s):=\left(\frac{1}{2}-s\right)\log{\frac{(1-s)i}{2\pi}}+s+O(e^{-t})+O\left(\frac{1}{|s|}\right)$ and differentiate it with respect to $s$ to obtain
\begin{align*}
f'(s) = -\log{\frac{(1-s)i}{2\pi}} + \frac{1}{2(1-s)} + O(e^{-t}) + O\left(\frac{1}{|s|^2}\right)
\end{align*}
and
\begin{equation*}
f^{(j)}(s) = O(e^{-t}) + O\left(\frac{1}{|s|^{j-1}}\right)
\end{equation*}
for $j\geq2$.
\eqref{eq:Fk2} immediately follows.
As a consequence to \eqref{eq:Fk2},
\begin{align*}
\frac{F^{(k)}}{F}(b+it)
&= \left(-\log{\frac{t+(1-b)i}{2\pi}} + \frac{1}{2(1-b-it)} + O\left(\frac{1}{t^2}\right)\right)^k \\
&\quad\times \left(1+ O\left(\frac{1}{t(\log{t})^2}\right)\right) \\
&= \left(-\log{\frac{t}{2\pi}} + \frac{t^2-2(1-b)((1-b)^2+t^2)}{2((1-b)^2+t^2)t}i + O\left(\frac{1}{t^2}\right)\right)^k \\
&\quad\times \left(1+ O\left(\frac{1}{t(\log{t})^2}\right)\right).
\end{align*}
This gives us
\begin{align*}
\log{\frac{F^{(k)}}{F}(b+it)} &= k\log{\log{\frac{t}{2\pi}}} + k\log{(-1)} \\
&\quad+ k\log\left(1-\frac{t^2-2(1-b)((1-b)^2+t^2)}{2((1-b)^2+t^2)t\log{\frac{t}{2\pi}}}i + O\left(\frac{1}{t^2\log{t}}\right)\right) \\
&\quad+  O\left(\frac{1}{t(\log{t})^2}\right) \\
&= k\log{\log{\frac{t}{2\pi}}} + k\log{(-1)} - k\frac{t^2-2(1-b)((1-b)^2+t^2)}{2((1-b)^2+t^2)t\log{\frac{t}{2\pi}}}i \\
&\quad+  O\left(\frac{1}{t(\log{t})^2}\right).
\end{align*}
Consequently we have
\begin{align*}
\text{Re}\left(\log{\frac{F^{(k)}}{F}(b+it)}\right) = k\log{\log{\frac{t}{2\pi}}} + O\left(\frac{1}{t(\log{t})^2}\right).
\end{align*}
Hence,
\begin{equation*}
\begin{aligned}
I_{13} &= \int_{t_k}^T\log{\left|\frac{F^{(k)}}{F}(b+it)\right|}dt =  \int_{t_k}^T\text{Re}\left(\log{\frac{F^{(k)}}{F}(b+it)}\right)dt\\
&= k\int_{t_k}^T \log{\log{\frac{t}{2\pi}}}dt + O\left(\int_{t_k}^T\frac{dt}{t(\log{t})^2}\right)\\
&= kT\log{\log{\frac{T}{2\pi}}} - 2\pi k\operatorname{Li}\left(\frac{T}{2\pi}\right) + O_{t_k}(1).
\end{aligned}
\end{equation*}

\vspace{2mm}
Finally, we estimate $I_{15}$.
Again from the functional equation for $\zeta(s)$ (see \eqref{eq:fe}), we have
\begin{align*}
\zeta^{(k)}(s) &= F^{(k)}(s)\zeta(1-s)\left(1 - \sum_{j=1}^k {{k}\choose{j}} (-1)^{j-1} \frac{1}{\frac{F^{(k)}}{F^{(k-j)}}(s)}\frac{\zeta^{(j)}}{\zeta}(1-s)\right)\\
&= \frac{F^{(k)}}{F}(s)\zeta(s)\left(1 - \sum_{j=1}^k {{k}\choose{j}} (-1)^{j-1} \frac{1}{\frac{F^{(k)}}{F^{(k-j)}}(s)}\frac{\zeta^{(j)}}{\zeta}(1-s)\right)
\end{align*}
which gives us
\begin{equation} \label{eq:fe2}
\frac{1}{\frac{F^{(k)}}{F}(s)}\frac{\zeta^{(k)}}{\zeta}(s) = 1 - \sum_{j=1}^k {{k}\choose{j}} (-1)^{j-1} \frac{1}{\frac{F^{(k)}}{F^{(k-j)}}(s)}\frac{\zeta^{(j)}}{\zeta}(1-s).
\end{equation}

It follows from condition 2 of Lemma \ref{lem1} that the right hand side of \eqref{eq:fe2} is holomorphic and has no zeros in the region defined by $\sigma\leq\sigma_k$ and $t\geq2$.
Moreover from conditions 3 and 4 of Lemma \ref{lem1}, the left hand side of \eqref{eq:fe2} is holomorphic and has no zeros in the region defined by $\sigma_k\leq\sigma<\frac{1}{2}$ and $t\geq t_k-1$.
Thus, we can determine $\log{\left(1 - \sum_{j=1}^k {{k}\choose{j}} (-1)^{j-1} \frac{1}{\frac{F^{(k)}}{F^{(k-j)}}(s)}\frac{\zeta^{(j)}}{\zeta}(1-s)\right)}$ so that it tends to $0$ as $\sigma\rightarrow-\infty$ which follows from condition 2 of Lemma \ref{lem1}, and is holomorphic in the region $\sigma<\frac{1}{2},\, t>t_k-1$.

Now we consider the trapezoid $C$ with vertices $b+it_k$, $b+iT$, $-T+iT$, and $-t_k+it_k$ (as in \cite[p. 2247]{aka}). Then by Cauchy's integral theorem,
\begin{equation} \label{eq:i15}
\int_C \log{\left(1 -  \sum_{j=1}^k {{k}\choose{j}} (-1)^{j-1} \frac{1}{\frac{F^{(k)}}{F^{(k-j)}}(s)}\frac{\zeta^{(j)}}{\zeta}(1-s)\right)}ds =0.
\end{equation}
By using condition 2 of Lemma \ref{lem1}, we can also show that (cf. \cite[p. 2248]{aka})
\begin{align*}
&\left(\int_{\sigma_k+iT}^{-T+iT} + \int_{-T+iT}^{-t_k+it_k} + \int_{-t_k+it_k}^{\sigma_k+it_k}\right) \\
&\quad\quad\quad\quad\quad\quad\quad
\log{\left(1 - \sum_{j=1}^k {{k}\choose{j}} (-1)^{j-1} \frac{1}{\frac{F^{(k)}}{F^{(k-j)}}(s)}\frac{\zeta^{(j)}}{\zeta}(1-s)\right)}ds = O(1).
\end{align*}
Next we estimate the integral from $\sigma_k+it_k$ to $b+it_k$ trivially and we obtain
\begin{align*}
\int_{\sigma_k+it_k}^{b+it_k} \log{\left(1 - \sum_{j=1}^k {{k}\choose{j}} (-1)^{j-1} \frac{1}{\frac{F^{(k)}}{F^{(k-j)}}(s)}\frac{\zeta^{(j)}}{\zeta}(1-s)\right)} ds = O_{t_k}(1).
\end{align*}

Substituting the above two equations into \eqref{eq:i15} and taking the imaginary part, we obtain
\begin{align*}
I_{15} &=  \int_{t_k}^T \log{\left|1 - \sum_{j=1}^k {{k}\choose{j}} (-1)^{j-1} \frac{1}{\frac{F^{(k)}}{F^{(k-j)}}(b+it)}\frac{\zeta^{(j)}}{\zeta}(1-b-it)\right|} dt \\
&= \int_{\sigma_k}^b \arg{\left(1 - \sum_{j=1}^k {{k}\choose{j}} (-1)^{j-1} \frac{1}{\frac{F^{(k)}}{F^{(k-j)}}(\sigma+iT)}\frac{\zeta^{(j)}}{\zeta}(1-\sigma-iT)\right)} d\sigma + O_{t_k}(1) \\
&\overset{\eqref{eq:fe2}}= \int_{\sigma_k}^b \arg{\left(\frac{1}{\frac{F^{(k)}}{F}(\sigma+iT)}\frac{\zeta^{(k)}}{\zeta}(\sigma+iT)\right)}d\sigma + O_{t_k}(1).
\end{align*}

Now we determine the logarithmic branch of $\log{\frac{F^{(k)}}{F}(s)}$ and $\log{\frac{\zeta^{(k)}}{\zeta}(s)}$ in the region $\sigma_k\leq\sigma<\frac{1}{2},\, t\geq t_k-1$ as in conditions 3 and 4, respectively, of Lemma \ref{lem1}.
Note that
\begin{align*}
\log{\left|\frac{1}{\frac{F^{(k)}}{F}(s)}\frac{\zeta^{(k)}}{\zeta}(s)\right|} = -\log{\left|\frac{F^{(k)}}{F}(s)\right|} + \log{\left|\frac{\zeta^{(k)}}{\zeta}(s)\right|}
\end{align*}
holds in the region $\sigma_k\leq\sigma<\frac{1}{2},\, t\geq t_k-1$.
Furthermore, since $\log{\left(\frac{1}{\frac{F^{(k)}}{F}(s)}\frac{\zeta^{(k)}}{\zeta}(s)\right)}$ ($= \log{\left(1 - \sum_{j=1}^k {{k}\choose{j}} (-1)^{j-1} \frac{1}{\frac{F^{(k)}}{F^{(k-j)}}(s)}\frac{\zeta^{(j)}}{\zeta}(1-s)\right)}$), $\log{\frac{F^{(k)}}{F}(s)}$, and $\log{\frac{\zeta^{(k)}}{\zeta}(s)}$ are holomorphic in this region, we know that $\arg{\left(\frac{1}{\frac{F^{(k)}}{F}(s)}\frac{\zeta^{(k)}}{\zeta}(s)\right)}$, $\arg{\frac{F^{(k)}}{F}(s)}$, and $\arg{\frac{\zeta^{(k)}}{\zeta}(s)}$ are continuous there. Since the region $\sigma_k\leq\sigma<\frac{1}{2},\, t\geq t_k-1$ is connected, there exists a constant $n\in\mathbb{Z}$ such that
\begin{equation*}
\arg{\left(\frac{1}{\frac{F^{(k)}}{F}(s)}\frac{\zeta^{(k)}}{\zeta}(s)\right)} = -\arg{\frac{F^{(k)}}{F}(s)} + \arg{\frac{\zeta^{(k)}}{\zeta}(s)} + 2n\pi
\end{equation*}
holds in $\sigma_k\leq\sigma<\frac{1}{2},\, t\geq t_k-1$.

From this choice of logarithmic branch, we have
\begin{equation} \label{eq:arg}
\frac{(3k-\beta_k)}{6}\pi+2n\pi < \arg{\left(\frac{1}{\frac{F^{(k)}}{F}(\sigma+iT)}\frac{\zeta^{(k)}}{\zeta}(\sigma+iT)\right)} < \frac{(9k-\alpha_k)}{6}\pi + 2n\pi
\end{equation}
for $\sigma_k\leq\sigma<\frac{1}{2}$.
Here, $\alpha_k$ and $\beta_k$  are the constants given in Lemma \ref{lem1}, that is,
\begin{align*}
(\alpha_k,\beta_k) =
\begin{cases}
(5,7), &\quad \text{if}\,\, k \,\,\text{is odd},\\
(-1,1), &\quad \text{if}\,\, k \,\,\text{is even}.
\end{cases}
\end{align*}
Since $n$ does not depend on $s$, $n=O_k(1)$. Therefore
\begin{align*}
\arg{\left(\frac{1}{\frac{F^{(k)}}{F}(\sigma+iT)}\frac{\zeta^{(k)}}{\zeta}(\sigma+iT)\right)} = O_k(1).
\end{align*}
From this, we can easily show that
\begin{equation*}
I_{15} = O_k(1),
\end{equation*}
for $\sigma_k$ and $t_k$ are fixed constants that depend only on $k$.

\vspace{2mm}
Inserting the estimates of $I_{12}$, $I_{13}$, $I_{14}$, and $I_{15}$ into \eqref{eq:i1}, we obtain
\begin{align*}
I_1 = (b\log{2}-k\log{\log{2}})T + \left(\frac{1}{2}-b\right)\left(T\log{\frac{T}{2\pi}}-T\right) &+ kT\log{\log{\frac{T}{2\pi}}} - 2k\pi\operatorname{Li}\left(\frac{T}{2\pi}\right)\\
&- \int_{1-b}^{a_k}\arg{\zeta(\sigma+iT)}d\sigma + O_k(1),
\end{align*}
since $a_k$ and $t_k$ are fixed constants that depend only on $k$.

\vspace{2mm}
To finalize the proof of Proposition \ref{prop2}, we insert the estimates of $I_1$, $I_2$, and $I_3$ into \eqref{eq:lit} to obtain
\begin{align*}
2\pi\sum_{\substack{\rho^{(k)}=\beta^{(k)}+i\gamma^{(k)},\\ 0<\gamma^{(k)}\leq T}} (\beta^{(k)}-b) &=
kT\log{\log{\frac{T}{2\pi}}} + (b\log{2}-k\log{\log{2}})T - 2k\pi\operatorname{Li}\left(\frac{T}{2\pi}\right)\\
&\quad+\left(\frac{1}{2}-b\right)\left(T\log{\frac{T}{2\pi}}-T\right) - \int_{1-b}^{a_k}\arg{\zeta(\sigma+iT)}d\sigma \\
&\quad+ \int_b^{a_k}\arg{G_k(\sigma+iT)}d\sigma + O_k(1).
\end{align*}
Taking the limit $\delta\rightarrow0$, we have $b\rightarrow\frac{1}{2}$, thus
\begin{align*}
\sum_{\substack{\rho^{(k)}=\beta^{(k)}+i\gamma^{(k)},\\ 0<\gamma^{(k)}\leq T}} \left(\beta^{(k)}-\frac{1}{2}\right) &= \frac{kT}{2\pi}\log{\log{\frac{T}{2\pi}}} + \frac{1}{2\pi}\left(\frac{1}{2}\log{2}-k\log{\log{2}}\right)T - k\operatorname{Li}\left(\frac{T}{2\pi}\right)\\
&\quad+ \frac{1}{2\pi}\int_{\frac{1}{2}}^{a_k}(-\arg{\zeta(\sigma+iT)}+\arg{G_k(\sigma+iT)})d\sigma + O_k(1).
\end{align*}
\end{proof}

\begin{rmk}
The proof of Proposition \ref{prop2} (and thus of Proposition 2.2 of \cite{aka}) actually, more or less, follows the proof of Theorem 10 given in \cite[Section 3]{lev}.
One obvious difference is that we did not estimate the fourth integral in \eqref{eq:lit} while Levinson and Montgomery estimated the corresponding integral (the fourth integral in (3.1) of \cite[Section 3]{lev}) as $O(\log{T})$. As in Akatsuka's \cite{aka} did, it turns out that this term contributes to the integral appearing in Proposition \ref{prop2} which will be estimated in the following few lemmas. This integral will contribute to the error term in Theorem \ref{cha1} and in the proofs of the following lemmas, we shall use the assumption of RH to reduce the upper bound of this integral.
\end{rmk}

\begin{rmk}
In contrast to the proof of Theorem 10 of \cite{lev}, in this paper (and in \cite{aka} as well), we describe some important estimates, such as those on $G_k(s)$, $\frac{F^{(k)}}{F}(s)$, and $\frac{\zeta^{(k)}}{\zeta}(s)$, which are related to the existence of fixed constants $a_k$, $\sigma_k$, and $t_k$ in Lemma \ref{lem1} for the sake of clarity. Furthermore, we also explicitly state Proposition \ref{prop2} since it clearly points out the main terms of Theorem \ref{cha1} and thus this gives the readers clear information of the term that in the current research contributes to the error term which is to be possibly improved in future research.
\end{rmk}

To complete the proof of Theorem \ref{cha1}, we need to estimate
\begin{align*}
\int_{\frac{1}{2}}^{a_k}(-\arg{\zeta(\sigma+iT)}+\arg{G_k(\sigma+iT)})d\sigma
\end{align*}
in Proposition \ref{prop2}.
For that purpose, similar to the method taken in \cite{aka}, below we give two bounds for $-\arg{\zeta(\sigma+iT)}+\arg{G_k(\sigma+iT)}$.
We write
\begin{align*}
-\arg{\zeta(\sigma+iT)}+\arg{G_k(\sigma+iT)} = \arg{\frac{G_k}{\zeta}(\sigma+iT)}
\end{align*}
where the argument on the right hand side is taken so that
$\log{\frac{G_k}{\zeta}(s)}$ tends to $0$ as $\sigma\rightarrow\infty$ and is holomorphic in $\mathbb{C}\backslash\{z+\lambda \,|\, (\zeta^{(k)}/\zeta)(z)=0 \,\,\text{or}\,\, \infty, \, \lambda\leq0\}$.

\vspace{3mm}
\begin{lem} \label{lem3}
Assume RH and let $T\geq t_k$.
Then for any $\epsilon_0>0$ satisfying $\epsilon_0<\frac{1}{2\log{T}}$
(since $T\geq t_k\geq100$, $\epsilon_0<\frac{1}{8}$),
we have for $\frac{1}{2}+\epsilon_0<\sigma\leq a_k$,
\begin{align*}
\arg{\frac{G_k}{\zeta}(\sigma+iT)} = O_{a_k,t_k} \left(\frac{\log{\frac{\log{T}}{\epsilon_0}}}{\sigma-\frac{1}{2}-\epsilon_0}\right).
\end{align*}
\end{lem}

\begin{proof}
To begin with, we note that $\frac{G_k}{\zeta}(s)$ is uniformly convergent to $1$ as $\sigma\rightarrow\infty$ for $t\in\mathbb{R}$, so we can take a number $c_k\in\mathbb{R}$ satisfying $a_k+1\leq c_k\leq\frac{t_k}{2}$ and $\frac{1}{2}\leq\text{Re}\left(\frac{G_k}{\zeta}(s)\right)\leq\frac{3}{2}\,\, (\text{when}\,\, \sigma\geq c_k)$. In fact, we can check that taking $c_k=10+k^2$ is enough.

\vspace{2mm}
The proof also proceeds similarly to the proof of Lemma 2.3 of \cite{aka}.
We let $\sigma\in(1/2+\epsilon_0,a_k]$ and let
$q_{G_k/\zeta}=q_{G_k/\zeta}(\sigma,T)$ denote the number of times $\text{Re}\left(\frac{G_k}{\zeta}(u+iT)\right)$ vanishes in $u\in[\sigma,c_k]$.
Then, $\left|\arg{\frac{G_k}{\zeta}(\sigma+iT)}\right|\leq\left(q_{G_k/\zeta}+1\right)\pi$.

\vspace{2mm}
Now we estimate $q_{G_k/\zeta}$.
For that purpose, we set
\begin{align*}
H_k(z)=H_{k_T}(z):=\frac{\frac{G_k}{\zeta}(z+iT)+\frac{G_k}{\zeta}(z-iT)}{2} \quad (z\in\mathbb{C})
\end{align*}
and $n_{H_k}(r):=\sharp\{z\in\mathbb{C} \,|\, H_k(z)=0, |z-c_k|\leq r\}$.
Then, we have $q_{G_k/\zeta}\leq n_{H_k}(c_k-\sigma)$ for $\frac{1}{2}+\epsilon_0<\sigma\leq a_k$.
For each $\sigma\in(1/2+\epsilon_0,a_k]$, we take $\epsilon=\epsilon_{\sigma,T}$ satisfying $0<\epsilon<\sigma-\frac{1}{2}-\epsilon_0$, then $H_k(z)$ is holomorphic in the region $\{z\in\mathbb{C} \,|\, |z-c_k| \leq c_k-\sigma+\epsilon\}$.
As in \cite[p. 2250]{aka}, by using Jensen's theorem (cf. \cite[pp. 125--126]{tit1}), we can show that
\begin{align*}
n_{H_k}(c_k-\sigma) &\leq \frac{1}{C_1\epsilon}\int_0^{c_k-\sigma+\epsilon}\frac{n_{H_k}(r)}{r}dr\\
&= \frac{1}{C_1\epsilon}\frac{1}{2\pi}\int_0^{2\pi}\log{|H_k(c_k+(c_k-\sigma+\epsilon)e^{i\theta})|}d\theta - \frac{1}{C_1\epsilon}\log{|H_k(c_k)|}
\end{align*}
for some constant $C_1>0$,
which by our choice of $c_k$ gives us
\begin{equation} \label{eq:nhk}
n_{H_k}(c_k-\sigma) \leq \frac{1}{C_1\epsilon}\frac{1}{2\pi}\int_0^{2\pi}\log{|H_k(c_k+(c_k-\sigma+\epsilon)e^{i\theta})|}d\theta + \frac{1}{\epsilon}O_{a_k,t_k}(1).
\end{equation}

\vspace{2mm}
Finally we estimate $\frac{1}{2\pi}\int_0^{2\pi}\log{|H_k(c_k+(c_k-\sigma+\epsilon)e^{i\theta})|}d\theta$.
From \cite[Theorems 9.2 and 9.6(A)]{tit2} (similar to what stated in \cite[p. 2250]{aka}),
\begin{align*}
\frac{\zeta'}{\zeta}(\sigma\pm it) = O\left(\frac{\log{T}}{\sigma-\frac{1}{2}}\right)
\end{align*}
holds for $\frac{1}{2}<\sigma\leq 2c_k$ and $\frac{T}{2}\leq t\leq 2T$.
Thus, for $\frac{1}{2}+\epsilon_0<\sigma\leq 2c_k$ and $\frac{T}{2}\leq t\leq 2T$, we have
\begin{equation} \label{eq:k1}
\frac{\zeta'}{\zeta}(\sigma\pm it) = O\left(\frac{\log{T}}{\epsilon_0}\right).
\end{equation}

\vspace{2mm}
With this estimate, we show that
\begin{align*}
\frac{\zeta^{(k)}}{\zeta}(s) = O\left(\frac{(\log{T})^k}{\epsilon_0^k}\right)
\end{align*}
holds for $\frac{1}{2}+\epsilon_0<\sigma<2c_k$ and $\frac{T}{2}\leq |t|\leq 2T$.
We use induction on $k$ in the equation.
For $k=1$, $\frac{\zeta'}{\zeta}(\sigma\pm it) = O\left(\frac{\log{T}}{\epsilon_0}\right)$ follows from \eqref{eq:k1}.
Suppose that $\frac{\zeta^{(n)}}{\zeta}(s) = O\left(\frac{(\log{T})^n}{\epsilon_0^n}\right)$ hold in the region $\frac{1}{2}+\epsilon_0<\sigma<2c_k,\, \frac{T}{2}\leq |t|\leq 2T$ for a positive integer $n$, then
\begin{equation} \label{eq:prodr}
\left(\frac{\zeta^{(n)}}{\zeta}(s)\right)' =  \frac{1}{2\pi i}\int_{|z-s|=\epsilon_0} \frac{\frac{\zeta^{(n)}}{\zeta}(z)}{(z-s)^2}dz = O\left(\frac{(\log{T})^n}{\epsilon_0^{n+1}}\right).
\end{equation}
Meanwhile,
\begin{align*}
\left(\frac{\zeta^{(n)}}{\zeta}(s)\right)' = \frac{\zeta^{(n+1)}}{\zeta}(s) - \frac{\zeta^{(n)}}{\zeta}(s)\frac{\zeta'}{\zeta}(s)
\end{align*}
holds in the region.

Therefore, by \eqref{eq:prodr} and by the induction hypothesis,
\begin{align*}
\frac{\zeta^{(n+1)}}{\zeta}(s) &= \left(\frac{\zeta^{(n)}}{\zeta}(s)\right)' + \frac{\zeta^{(n)}}{\zeta}(s)\frac{\zeta'}{\zeta}(s) = O\left(\frac{(\log{T})^n}{\epsilon_0^{n+1}}\right) + O\left(\frac{(\log{T})^{n+1}}{\epsilon_0^{n+1}}\right)\\
&= O\left(\frac{(\log{T})^{n+1}}{\epsilon_0^{n+1}}\right)
\end{align*}
holds for $\frac{1}{2}<\sigma\leq 2c_k$ and $\frac{T}{2}\leq |t|\leq 2T$.
Hence, by induction, we find that
\begin{align*}
\frac{\zeta^{(k)}}{\zeta}(s) = O\left(\frac{(\log{T})^k}{\epsilon_0^k}\right)
\end{align*}
holds in the region defined by $\frac{1}{2}+\epsilon_0<\sigma<2c_k$ and $\frac{T}{2}\leq |t|\leq 2T$.
This immediately gives us
\begin{align*}
|H_k(c_k+(c_k-\sigma+\epsilon)e^{i\theta})| \ll_{a_k,t_k} \frac{(\log{T})^k}{\epsilon_0^k},
\end{align*}
and so
\begin{equation*}
|H_k(c_k+(c_k-\sigma+\epsilon)e^{i\theta})| \leq C_2(a_k,t_k) \frac{(\log{T})^k}{\epsilon_0^k}
\end{equation*}
for some constant $C_2>0$ which depends only on $a_k$ and $t_k$.
Thus,
\begin{align*}
\frac{1}{2\pi}\int_0^{2\pi}\log{|H_k(c_k+(c_k-\sigma+\epsilon)e^{i\theta})|}d\theta \leq \log{C_2(a_k,t_k)} + k\log{\frac{\log{T}}{\epsilon_0}} \ll_{a_k,t_k} \log{\frac{\log{T}}{\epsilon_0}}.
\end{align*}
Applying this to \eqref{eq:nhk}, we obtain
\begin{align*}
n_{H_k}(c_k-\sigma) = \frac{1}{\epsilon}O_{a_k,t_k}\left(\log{\frac{\log{T}}{\epsilon_0}}\right)
\end{align*}
which implies
\begin{equation*}
\arg{\frac{G_k}{\zeta}(\sigma+iT)} = \frac{1}{\epsilon}O_{a_k,t_k}\left(\log{\frac{\log{T}}{\epsilon_0}}\right).
\end{equation*}
Taking $\epsilon=\frac{1}{2}\left(\sigma-\frac{1}{2}-\epsilon_0\right) \quad \left(<\sigma-\frac{1}{2}-\epsilon_0\right)$, we obtain
\begin{align*}
\arg{\frac{G_k}{\zeta}(\sigma+iT)} = O_{a_k,t_k}\left(\frac{\log{\frac{\log{T}}{\epsilon_0}}}{\sigma-\frac{1}{2}-\epsilon_0}\right).
\end{align*}
\end{proof}

\begin{lem} \label{lem4}
Assume RH and let $A\geq2$ be fixed. Then there exists a constant $C_0>0$ such that
\begin{equation*}
\left|\zeta^{(k)}(\sigma+it)\right|\leq\exp{\left( C_0\left(\frac{(\log T)^{2(1-\sigma)}}{\log\log{T}}+(\log{T})^{1/10}\right)\right)}
\end{equation*}
holds for $T\geq t_k$, $\frac{T}{2}\leq t\leq 2T$, $\frac{1}{2}-\frac{1}{\log{\log{T}}}\leq\sigma\leq A$.
\end{lem}

\begin{proof}
Referring to \cite[(14.14.2), (14.14.5), and the first equation on p. 384]{tit2} (cf. \cite[pp. 2251--2252]{aka}), we know that
\begin{equation} \label{eq:lem4zen}
|\zeta(\sigma+it)| \leq \exp{\left(C_3\left(\frac{(\log{T})^{2(1-\sigma)}}{\log{\log{T}}}\right)+(\log{T})^{1/10}\right)}
\end{equation}
holds for  $\frac{1}{2}-\frac{2}{\log{\log{T}}}\leq\sigma\leq A+1$, $\frac{T}{3}\leq t\leq 3T$ for some constant $C_3>0$.

Applying Cauchy's integral formula, we see that
\begin{align*}
\zeta^{(k)}(s) = \frac{k!}{2\pi i} \int_{|z-s|=\epsilon} \frac{\zeta(z)}{(z-s)^{k+1}}dz \quad\quad \text{for} \quad 0<\epsilon<\frac{1}{2}
\end{align*}
holds in the region defined by  $\frac{1}{2}-\frac{1}{\log{\log{T}}}\leq\sigma\leq A$ and $\frac{T}{2}\leq t\leq 2T$.
Applying \eqref{eq:lem4zen} and by taking $\epsilon=\frac{1}{2(\log{\log{T}})^{1/k}} \,\, (<\frac{1}{2})$, we obtain Lemma \ref{lem4}.\\
\end{proof}

\begin{lem} \label{lem5}
Assume RH and let $T\geq t_k$.
Then for any $\frac{1}{2}\leq\sigma\leq\frac{3}{4}$, we have
\begin{align*}
\arg{G_k(\sigma+iT)} = O_{a_k}\left(\frac{(\log T)^{2(1-\sigma)}}{(\log{\log{T}})^{1/2}}\right).
\end{align*}
\end{lem}

\begin{proof}
The proof proceeds in the same way as the proof of Lemma 2.4 of \cite{aka}.
Refer to \cite[pp. 2252--2253]{aka} for the detailed proof and use Lemma \ref{lem4} above in place of Lemma 2.6 of \cite{aka}.\\
\end{proof}

\begin{rmk}
The restrictions of the lower bound of $T$ we gave in Lemmas \ref{lem3}, \ref{lem4}, and \ref{lem5} are not essential, but they are sufficient for our needs.
We may let $T$ be any positive number in Lemmas \ref{lem3}, \ref{lem4}, and \ref{lem5}, however in that case, we need to modify some calculations in the proofs. Thus we used these restrictions for our convenience.\\
\end{rmk}

\begin{proof}[\textbf{Proof of Theorem \ref{cha1}}]

First of all, we consider for $T\geq t_k$ which satisfies $\zeta^{(k)}(\sigma+iT)\neq0$ and $\zeta(\sigma+iT)\neq0$ for any $\sigma\in\mathbb{R}$. By Lemma \ref{lem3}, we have
\begin{align*}
\int_{\frac{1}{2}+2\epsilon_0}^{a_k} \arg{\frac{G_k}{\zeta}(\sigma+iT)}d\sigma \ll_{a_k,t_k} \int_{\frac{1}{2}+2\epsilon_0}^{a_k} \frac{\log{\frac{\log{T}}{\epsilon_0}}}{\sigma-\frac{1}{2}-\epsilon_0}d\sigma
\ll_{a_k} \log{\frac{\log{T}}{\epsilon_0}}\log{\frac{1}{\epsilon_0}}.
\end{align*}

\vspace{2mm}
Next, by Lemma \ref{lem5},
\begin{align*}
\arg{G_k(\sigma+iT)} = O_{a_k}\left(\frac{(\log T)^{2(1-\sigma)}}{(\log{\log{T}})^{1/2}}\right) \quad\quad \text{for} \quad \frac{1}{2}\leq\sigma\leq\frac{3}{4}
\end{align*}
and from (2.23) of \cite[p. 2251]{aka} (cf. \cite[(14.14.3) and (14.14.5)]{tit2}),
RH implies that
\begin{equation*}
\arg{\zeta(\sigma+iT)} = O\left(\frac{(\log T)^{2(1-\sigma)}}{\log{\log{T}}}\right)
\end{equation*}
holds uniformly for $\frac{1}{2}\leq\sigma\leq\frac{3}{4}$.
Thus,
\begin{align*}
\int_{\frac{1}{2}}^{\frac{1}{2}+2\epsilon_0} \arg{\frac{G_k}{\zeta}(\sigma+iT)}d\sigma \ll_{a_k} \frac{\log T}{(\log{\log{T}})^{1/2}}\epsilon_0.
\end{align*}

\vspace{2mm}
Now we take $\epsilon_0 = \frac{1}{4\log{T}} \,\, \left(<\frac{1}{2}\frac{1}{\log{T}}\right)$, then we have
\begin{align*}
\int_{\frac{1}{2}}^{a_k} \arg{\frac{G_k}{\zeta}(\sigma+iT)}d\sigma \ll_{a_k,t_k} (\log{\log{T}})^2.
\end{align*}
Applying this to Proposition \ref{prop2} and noting that $a_k$ and $t_k$ are fixed constants that depend only on $k$, we have
\begin{equation} \label{eq:kari1}
\begin{aligned}
\sum_{\substack{\rho^{(k)}=\beta^{(k)}+i\gamma^{(k)},\\ 0<\gamma^{(k)}\leq T}} \left(\beta^{(k)}-\frac{1}{2}\right) = \,\frac{kT}{2\pi}\log{\log{\frac{T}{2\pi}}} &+ \frac{1}{2\pi}\left(\frac{1}{2}\log{2} - k\log{\log{2}}\right)T - k\operatorname{Li}\left(\frac{T}{2\pi}\right)\\
&\quad\quad\quad\quad\quad\quad\quad\quad\quad+ O_k((\log{\log{T}})^2).
\end{aligned}
\end{equation}

\vspace{3mm}
Secondly, for $2\pi<T<t_k$, we are adding some finite number of terms which depend on $t_k$, and thus depend only on $k$ so this can be included in the error term.\\

Thirdly, for $T\geq t_k$ such that $\zeta^{(k)}(\sigma+iT)=0$ or $\zeta(\sigma+iT)=0$ for some $\sigma\in\mathbb{R}$, there is some increment in the value of $\sum_{\substack{\rho^{(k)}=\beta^{(k)}+i\gamma^{(k)},\\ 0<\gamma^{(k)}\leq T}} \left(\beta^{(k)}-\frac{1}{2}\right)$ as much as
$$
\sum_{\substack{\rho^{(k)}=\beta^{(k)}+i\gamma^{(k)},\\ \gamma^{(k)} = T}} \left(\beta^{(k)}-\frac{1}{2}\right).
$$
Now we estimate this and we show that this can be included in the error term of \eqref{eq:kari1}.
We start by taking a small $0<\epsilon<1$ such that $\zeta^{(k)}(\sigma+i(T\pm\epsilon))\neq0$ and $\zeta(\sigma+i(T\pm\epsilon))\neq0$ for any $\sigma\in\mathbb{R}$.
According to \eqref{eq:kari1},
\begin{align*}
\sum_{\substack{\rho^{(k)}=\beta^{(k)}+i\gamma^{(k)},\\ 0<\gamma^{(k)}\leq T+\epsilon}} \left(\beta^{(k)}-\frac{1}{2}\right) = \frac{k(T+\epsilon)}{2\pi}\log{\log{\frac{T+\epsilon}{2\pi}}} &+ \frac{1}{2\pi}\left(\frac{1}{2}\log{2} - k\log{\log{2}}\right)(T+\epsilon)\\
&- k\operatorname{Li}\left(\frac{T+\epsilon}{2\pi}\right) + O_k((\log{\log{T}})^2),
\end{align*}
\begin{align*}
\sum_{\substack{\rho^{(k)}=\beta^{(k)}+i\gamma^{(k)},\\ 0<\gamma^{(k)}\leq T-\epsilon}} \left(\beta^{(k)}-\frac{1}{2}\right) = \frac{k(T-\epsilon)}{2\pi}\log{\log{\frac{T-\epsilon}{2\pi}}} &+ \frac{1}{2\pi}\left(\frac{1}{2}\log{2} - k\log{\log{2}}\right)(T-\epsilon)\\
&- k\operatorname{Li}\left(\frac{T-\epsilon}{2\pi}\right) + O_k((\log{\log{T}})^2).
\end{align*}
Thus,
\begin{align*}
\sum_{\substack{\rho^{(k)}=\beta^{(k)}+i\gamma^{(k)},\\ T-\epsilon<\gamma^{(k)}\leq T+\epsilon}} \left(\beta^{(k)}-\frac{1}{2}\right) &= \frac{k(T+\epsilon)}{2\pi}\log{\log{\frac{T+\epsilon}{2\pi}}} - \frac{k(T-\epsilon)}{2\pi}\log{\log{\frac{T-\epsilon}{2\pi}}} \\
&\quad+ \frac{\epsilon}{\pi}\left(\frac{1}{2}\log{2} - k\log{\log{2}}\right) \\
&\quad- k\left(\operatorname{Li}\left(\frac{T+\epsilon}{2\pi}\right) - \operatorname{Li}\left(\frac{T-\epsilon}{2\pi}\right)\right) + O_k((\log{\log{T}})^2) \\
&= \frac{k\epsilon}{\pi}\log{\log{\frac{T}{2\pi}}} + \frac{k\epsilon}{\pi\log{\frac{T}{2\pi}}} + \frac{\epsilon}{\pi}\left(\frac{1}{2}\log{2} - k\log{\log{2}}\right) \\
&\quad- \frac{k\epsilon}{\pi\log{\frac{T}{2\pi}}} + O\left(\frac{\epsilon^2}{T\log{T}}\right) + O_k((\log{\log{T}})^2) \\
&= \frac{k\epsilon}{\pi}\log{\log{\frac{T}{2\pi}}} + \frac{\epsilon}{\pi}\left(\frac{1}{2}\log{2} - k\log{\log{2}}\right) + O_k((\log{\log{T}})^2).
\end{align*}
This gives us
\begin{equation*}
\begin{aligned}
\sum_{\substack{\rho^{(k)}=\beta^{(k)}+i\gamma^{(k)},\\ T-\epsilon<\gamma^{(k)}\leq T+\epsilon}} \left(\beta^{(k)}-\frac{1}{2}\right) = O_k((\log{\log{T}})^2)
\end{aligned}
\end{equation*}
which implies
\begin{equation*}
\begin{aligned}
\sum_{\substack{\rho^{(k)}=\beta^{(k)}+i\gamma^{(k)},\\ \gamma^{(k)} = T}} \left(\beta^{(k)}-\frac{1}{2}\right) = O_k((\log{\log{T}})^2).
\end{aligned}
\end{equation*}
Therefore, this increment can also be included in the error term.\\

Hence,
\begin{align*}
\sum_{\substack{\rho^{(k)}=\beta^{(k)}+i\gamma^{(k)},\\ 0<\gamma^{(k)}\leq T}} \left(\beta^{(k)}-\frac{1}{2}\right) = \frac{kT}{2\pi}\log{\log{\frac{T}{2\pi}}} &+ \frac{1}{2\pi}\left(\frac{1}{2}\log{2} - k\log{\log{2}}\right)T \\
&- k\operatorname{Li}\left(\frac{T}{2\pi}\right) + O_k((\log{\log{T}})^2)
\end{align*}
holds for any $T>2\pi$.

\end{proof}

\begin{proof}[\textbf{Proof of Corollary \ref{cha12}}]

This is an immediate consequence of Theorem \ref{cha1}.
For the proof, refer to \cite[p. 58 (the ending part of Section 3)]{lev}.

\end{proof}


\section{Proof of Theorem \ref{cha2}}
\label{sec:3}

In this section we give the proof of Theorem \ref{cha2}.
We first show the following proposition.

\begin{prop} \label{prop6}
Assume RH. Then for $T\geq2$ which satisfies $\zeta(\sigma+iT)\neq0$ and $\zeta^{(k)}(\sigma+iT)\neq0$ for all $\sigma\in\mathbb{R}$, we have
\begin{align*}
N_k(T) = \frac{T}{2\pi}\log{\frac{T}{4\pi}}-\frac{T}{2\pi}+\frac{1}{2\pi}\arg{G_k\left(\frac{1}{2}+iT\right)}+\frac{1}{2\pi}\arg{\zeta\left(\frac{1}{2}+iT\right)}+O_k(1)
\end{align*}
where the arguments are taken as in Proposition \ref{prop2}.
\end{prop}

\begin{proof}
The steps of the proof also follow those of the proof of Proposition 3.1 of \cite{aka}.
We take $a_k$, $\sigma_k$, $t_k$, $T$, $\delta$, and $b$ as in the beginning of the proof of Proposition \ref{prop2}. We let $b' := \frac{1}{2}-\frac{\delta}{2}$. Replacing $b$ by $b'$ in \eqref{eq:lit}, we have
\begin{align*}
2\pi\sum_{\substack{\rho^{(k)}=\beta^{(k)}+i\gamma^{(k)},\\ 0<\gamma^{(k)}\leq T}} (\beta^{(k)}-b') &=
\int_{t_k}^T\log{|G_k(b'+it)|}dt - \int_{t_k}^T\log{|G_k(a_k+it)|}dt \\
&\quad- \int_{b'}^{a_k}\arg{G_k(\sigma+it_k)}d\sigma + \int_{b'}^{a_k}\arg{G_k(\sigma+iT)}d\sigma.
\end{align*}

Subtracting this from \eqref{eq:lit}, we have
\begin{align*}
\pi\delta (N_k(T) - N_k(t_k)) &=
\int_{t_k}^T\log{|G_k(b+it)|}dt - \int_{t_k}^T\log{|G_k(b'+it)|}dt\\
&\quad- \int_b^{b'}\arg{G_k(\sigma+it_k)}d\sigma + \int_b^{b'}\arg{G_k(\sigma+iT)}d\sigma\\
&=: J_1 + J_2 + J_3 + \int_b^{b'}\arg{G_k(\sigma+iT)}d\sigma. \numberthis \label{eq:lit'}
\end{align*}

\vspace{2mm}
Referring to the estimate of $I_3$ in the proof of Proposition \ref{prop2} (cf. \cite[p. 2246]{aka}), we can easily show that
\begin{equation*}
J_3 = O_{t_k}(\delta).
\end{equation*}

\vspace{2mm}
Now we estimate $J_1 + J_2$.
From \eqref{eq:i1}, we have
\begin{align*}
J_1+J_2 &= \int_{t_k}^T\log{|G_k(b+it)|}dt - \int_{t_k}^T\log{|G_k(b'+it)|}dt\\
&= ((b-b')\log{2})(T-t_k) + \int_{t_k}^T (\log{|F(b+it)|}-\log{|F(b'+it)|})dt\\
&\quad+ \int_{t_k}^T \left(\log{\left|\frac{F^{(k)}}{F}(b+it)\right|}-\log{\left|\frac{F^{(k)}}{F}(b'+it)\right|}\right)dt\\
&\quad+ \int_{t_k}^T (\log{|\zeta(1-b-it)|}-\log{|\zeta(1-b'-it)|})dt\\
&\quad+ \int_{t_k}^T \Bigg(\log{\left|1 -  \sum_{j=1}^k {{k}\choose{j}} (-1)^{j-1} \frac{1}{\frac{F^{(k)}}{F^{(k-j)}}(b+it)}\frac{\zeta^{(j)}}{\zeta}(1-b-it)\right|}\\
&\quad\quad\quad\quad\quad\quad- \log{\left|1 -  \sum_{j=1}^k {{k}\choose{j}} (-1)^{j-1} \frac{1}{\frac{F^{(k)}}{F^{(k-j)}}(b'+it)}\frac{\zeta^{(j)}}{\zeta}(1-b'-it)\right|}\Bigg)dt\\
&=: \left(\left(-\frac{\delta}{2}\log{2}\right)T + O_{t_k}(\delta)\right) + J_{12} + J_{13} + J_{14} + J_{15}.
\end{align*}

\vspace{2mm}
Referring to \cite[pp. 2255--2256]{aka}, we have
\begin{align*}
J_{12} = \frac{\delta}{2}\left(T\log{\frac{T}{2\pi}}-T\right) + O_{t_k}(\delta),
\end{align*}
\begin{equation*}
J_{14} = \int_{1-b'}^{1-b}\arg{\zeta(\sigma+iT)}d\sigma + O_{t_k}(\delta).
\end{equation*}

\vspace{2mm}
We only need to estimate $J_{13}$ and $J_{15}$.
We begin with the estimation of $J_{13}$. We determine the logarithmic branch of $\log{\frac{F^{(k)}}{F}(s)}$ for $0<\sigma<\frac{1}{2}$ and $t>t_k-1$ as in condition 3 of Lemma \ref{lem1}. We then have $\arg{\frac{F^{(k)}}{F}(s)}\in(\alpha_k\pi/6,\beta_k\pi/6)$, where the pair $(\alpha_k,\beta_k)$ is defined as in Lemma \ref{lem1}.

As in \cite[p. 2255]{aka}, we apply Cauchy's integral theorem to $\log{\frac{F^{(k)}}{F}(s)}$ on the rectangle with vertices $b+it_k$, $b'+it_k$, $b'+iT$, and $b+iT$ and take the imaginary part, then we obtain
\begin{align*}
J_{13} =  \int_b^{b'} \arg{\frac{F^{(k)}}{F}(\sigma+it_k)}d\sigma - \int_b^{b'} \arg{\frac{F^{(k)}}{F}(\sigma+iT)}d\sigma = O_k(\delta).
\end{align*}

\vspace{2mm}
Finally, we estimate $J_{15}$. 
We determine the logarithmic branch of
\begin{align*}
\log{\left(1 - \sum_{j=1}^k {{k}\choose{j}} (-1)^{j-1} \frac{1}{\frac{F^{(k)}}{F^{(k-j)}}(s)}\frac{\zeta^{(j)}}{\zeta}(1-s)\right)}
\end{align*}
in the same manner as that in the estimation of $I_{15}$ in the proof of Proposition \ref{prop2}, then it is holomorphic in the region $0<\sigma<\frac{1}{2},\, t>t_k-1$.
Applying Cauchy's integral theorem to it on the path taken for estimating $J_{13}$, we have
\begin{align*}
J_{15} &= \int_b^{b'} \arg{\left(1 - \sum_{j=1}^k {{k}\choose{j}} (-1)^{j-1} \frac{1}{\frac{F^{(k)}}{F^{(k-j)}}(\sigma+it_k)}\frac{\zeta^{(j)}}{\zeta}(1-\sigma-it_k)\right)} \\
&\quad- \int_b^{b'} \arg{\left(1 - \sum_{j=1}^k {{k}\choose{j}} (-1)^{j-1} \frac{1}{\frac{F^{(k)}}{F^{(k-j)}}(\sigma+iT)}\frac{\zeta^{(j)}}{\zeta}(1-\sigma-iT)\right)}d\sigma.
\end{align*}
Again using \eqref{eq:fe2}, 
\begin{align*}
J_{15} = \int_b^{b'} \arg{\left(\frac{1}{\frac{F^{(k)}}{F}(\sigma+it_k)}\frac{\zeta^{(k)}}{\zeta}(\sigma+it_k)\right)}d\sigma - \int_b^{b'} \arg{\left(\frac{1}{\frac{F^{(k)}}{F}(\sigma+iT)}\frac{\zeta^{(k)}}{\zeta}(\sigma+iT)\right)}d\sigma.
\end{align*}
Applying \eqref{eq:arg}, we obtain
\begin{align*}
J_{15} = O_k(\delta).
\end{align*}
Hence, since $t_k$ is a fixed constant that depends only on $k$,
\begin{align*}
J_1+J_2 = \frac{\delta}{2}\left(T\log{\frac{T}{4\pi}}-T\right) + \int_{1-b'}^{1-b} \arg{\zeta(\sigma+iT)}d\sigma + O_k(\delta).
\end{align*}

\vspace{3mm}
Inserting the estimates of $J_1+J_2$ and $J_3$ into \eqref{eq:lit'}, we have
\begin{equation} \label{eq:n2t}
N_k(T) = \frac{T}{2\pi}\log{\frac{T}{4\pi}} - \frac{T}{2\pi} + \frac{1}{\pi\delta}\left(\int_{1-b'}^{1-b} \arg{\zeta(\sigma+iT)}d\sigma + \int_b^{b'} \arg{G_k(\sigma+iT)}d\sigma \right)+ O_k(1).
\end{equation}

\vspace{3mm}
Taking the limit $\delta\rightarrow0$ and applying the mean value theorem,
\begin{align*}
\lim_{\delta\rightarrow0} \frac{1}{\pi\delta}\int_{1-b'}^{1-b} \arg{\zeta(\sigma+iT)}d\sigma = \frac{1}{2\pi}\arg{\zeta\left(\frac{1}{2}+iT\right)}
\end{align*}
by noting that $b=\frac{1}{2}-\delta$ and $b'=\frac{1}{2}-\frac{\delta}{2}$.
And similarly,
\begin{align*}
\lim_{\delta\rightarrow0} \frac{1}{\pi\delta}\int_b^{b'} \arg{G_k(\sigma+iT)}d\sigma = \frac{1}{2\pi}\arg{G_k\left(\frac{1}{2}+iT\right)}.
\end{align*}
Substituting these into \eqref{eq:n2t}, we have
\begin{align*}
N_k(T) = \frac{T}{2\pi}\log{\frac{T}{4\pi}} - \frac{T}{2\pi} + \frac{1}{2\pi}\arg{G_k\left(\frac{1}{2}+iT\right)} + \frac{1}{2\pi}\arg{\zeta\left(\frac{1}{2}+iT\right)} + O_k(1).
\end{align*}

\vspace{3mm}
If $2\leq T<t_k$, then $N_k(T)\leq N_k(t_k) = O_{t_k}(1) = O_k(1)$.
Hence the above equation holds for any $T\geq2$ which satisfies the conditions of Proposition \ref{prop6}.\\
\end{proof}

\begin{proof}[\textbf{Proof of Theorem \ref{cha2}}]

Firstly we consider for $T\geq3$ which satisfies $\zeta^{(k)}(\sigma+iT)\neq0$ and $\zeta(\sigma+iT)\neq0$ for any $\sigma\in\mathbb{R}$.
By Lemma \ref{lem5},
\begin{align*}
\arg{G_k\left(\frac{1}{2}+iT\right)} = O_{a_k}\left(\frac{\log T}{(\log{\log{T}})^{1/2}}\right)
\end{align*}
and again from equation (2.23) of \cite[p. 2251]{aka}, we have
\begin{align*}
\arg{\zeta\left(\frac{1}{2}+iT\right)} = O\left(\frac{\log T}{\log{\log{T}}}\right).
\end{align*}
Substituting these into Proposition \ref{prop6}, we obtain
\begin{align*}
N_k(T) = \frac{T}{2\pi}\log{\frac{T}{4\pi}} - \frac{T}{2\pi} + O_k\left(\frac{\log{T}}{(\log{\log{T}})^{1/2}}\right).
\end{align*}

Next, if $\zeta(\sigma+iT)=0$ or $\zeta^{(k)}(\sigma+iT)=0$ for some $\sigma\in\mathbb{R} \quad (T\geq3)$, then we again take a small $0<\epsilon<1$ such that $\zeta^{(k)}(\sigma+i(T\pm\epsilon))\neq0$ and $\zeta(\sigma+i(T\pm\epsilon))\neq0$ for any $\sigma\in\mathbb{R}$ as in the proof of Theorem \ref{cha1}.
Then similarly, we can show that the increment of the value of $N_k(T)$ can be included in the error term of the above equation.

Therefore
\begin{equation*}
N_k(T) = \frac{T}{2\pi}\log{\frac{T}{4\pi}} - \frac{T}{2\pi} + O_k\left(\frac{\log{T}}{(\log{\log{T}})^{1/2}}\right)
\end{equation*}
holds for any $T\geq3$.

\end{proof}


\begin{rmk}
It is well-known that in the case of the Riemann zeta function $\zeta(s)$, the number $N(T)$ of zeros of $\zeta(s)$ is estimated as
\begin{align*}
N(T) = \frac{T}{2\pi}\log{\frac{T}{2\pi}} - \frac{T}{2\pi} + S(T) + O\left(\frac{1}{T}\right),
\end{align*}
where $S(T) = \frac{1}{\pi}\arg{\zeta\left(\frac{1}{2}+iT\right)}$ with a standard branch (cf. \cite[Theorem 9.3]{tit2}).
Thus, the function $S(T)$ determines the error term in the estimate of $N(T)$.
Under RH, we have
\begin{align} \label{eq:st}
S(T) = O\left(\frac{\log{T}}{\log{\log{T}}}\right)
\end{align}
(cf. \cite[(14.13.1) of Theorem 14.13]{tit2}).
In comparison to the above estimate, the term that determines the error term of $N_k(T)$ is
\begin{align*}
\frac{1}{2\pi}\arg{G_k\left(\frac{1}{2}+iT\right)}+\frac{1}{2\pi}\arg{\zeta\left(\frac{1}{2}+iT\right)}
\end{align*}
by Proposition \ref{prop6} and under RH, they are currently estimated as follows:
\begin{align} \label{eq:skt}
\arg{G_k\left(\frac{1}{2}+iT\right)} =  O_k\left(\frac{\log{T}}{(\log{\log{T}})^{1/2}}\right),\quad \arg{\zeta\left(\frac{1}{2}+iT\right)} = O\left(\frac{\log{T}}{\log{\log{T}}}\right).
\end{align}
This estimate of $\arg{G_k\left(\frac{1}{2}+iT\right)}$ determines the error term of $N_k(T)$
and it results in $N_k(T)$ having error term slightly greater in magnitude than that of $N(T)$.
However, this is the best known estimate on $N_k(T)$ under RH at present.
\end{rmk}

Furthermore, the size of the implied $O$-constant in \eqref{eq:st} has been studied in many papers, such as \cite{car} and \cite{far}. In contrast to this, we currently have no information about the implied $O$-constant in the first equation of \eqref{eq:skt}.


\section*{Acknowledgements}

The author would like to express her gratitude to Prof. Kohji Matsumoto, Prof. Hirotaka Akatsuka, senior Ryo Tanaka, and also the referee for their valuable advice.



\begin{thebibliography}{1}
\bibitem{aka} H. Akatsuka, \emph{Conditional estimates for error terms related to the distribution of zeros of $\zeta'(s)$}, J. Number Theory {\bf 132} (2012), no. 10, 2242--2257.
\bibitem{ber} B. C. Berndt, \emph{The number of zeros for $\zeta^{(k)}(s)$}, J. Lond. Math. Soc. (2) {\bf 2} (1970), 577--580.
\bibitem{car} E. Carneiro, V. Chandee, M. B. Milinovich, \emph{Bounding $S(t)$ and $S_1(t)$ on the Riemann hypothesis}, Math. Ann. {\bf 356} (2013), no. 3, 939--968.
\bibitem{con} J. B. Conrey and A. Ghosh, \emph{Zeros of derivatives of the Riemann zeta-function near the critical line}, in "Analytic number theory, Proc. Conf. in Honor of P. T. Bateman  (Allerton Park, IL, USA, 1989)",  B. C. Berndt et al. (eds.), Progr. Math. Vol. 85, Birkh$\ddot{a}$user Boston, Boston, MA, 1990, pp. 95--110.
\bibitem{far} D. W. Farmer, S. M. Gonek, C. P. Hughes, \emph{The maximum size of $L$-functions}, J. Reine Angew. Math. {\bf 609} (2007), 215--236.
\bibitem{gon} S. M. Gonek, \emph{Mean values of the Riemann zeta-function and its derivatives}, Invent. Math. {\bf 75} (1984), no. 1, 123--141.
\bibitem{lev} N. Levinson and H. L. Montgomery, \emph{Zeros of the derivatives of the Riemann zeta-function}, Acta Math. {\bf 133} (1974), 49--65.
\bibitem{spe} A. Speiser, \emph{Geometrisches zur Riemannschen Zetafunktion}, Math. Ann. {\bf 110} (1935), no. 1, 514--521.
\bibitem{spi1} R. Spira, \emph{Another zero-free region for $\zeta^{(k)}(s)$}, Proc. Amer. Math. Soc. {\bf 26} (1970), 246--247.
\bibitem{spi2} R. Spira, \emph{Zero-free regions of $\zeta^{(k)}(s)$}, J. Lond. Math. Soc. {\bf 40} (1965), 677--682.
\bibitem{spi3} R. Spira, \emph{Zeros of $\zeta'(s)$ and the Riemann hypothesis}, Illinois J. Math. {\bf 17} (1973), 147--152.
\bibitem{ade} A. I. Suriajaya, \emph{On the Zeros of the Second Derivative of the Riemann Zeta Function under the Riemann Hypothesis}, arxiv:1309.7160 [math.NT].
\bibitem{tit1} E. C. Titchmarsh, \emph{The theory of functions}, second ed., Oxford Univ. Press, 1939.
\bibitem{tit2} E. C. Titchmarsh, \emph{The theory of the Riemann zeta-function}, second ed. (revised by D. R. Heath-Brown), Oxford Univ. Press, 1986.
\bibitem{yil1} C. Y. Yildirim, \emph{A Note on $\zeta''(s)$ and $\zeta'''(s)$}, Proc. Amer. Math. Soc. {\bf 124} (1996), no. 8, 2311--2314.
\bibitem{yil2} C. Y. Yildirim, \emph{Zeros of $\zeta''(s)$ \emph{\&} $\zeta'''(s)$ in $\sigma<\frac{1}{2}$}, Turk. J. Math. {\bf 24} (2000), no. 1, 89--108.
\end{thebibliography}
\end{document}